\documentclass[11pt]{article}
\usepackage{amsmath, amssymb, amsthm, enumerate}
\usepackage{fullpage}
\usepackage{color,ifpdf,latexsym,graphicx,url}
\usepackage{hyperref}

\newcommand\NN{\mathbb{N}}

\DeclareMathOperator\mex{mex}

\newtheorem{thm}{Theorem}[section]
\newtheorem{lemma}[thm]{Lemma}

\newtheorem{cor}[thm]{Corollary}

\title{Winning strategies for aperiodic subtraction games}
\author{%
	Alan Guo\\%
	\small MIT Computer Science and Artificial Intelligence Laboratory\\%
	\small Cambridge, MA 02139, USA\\%
	\small \texttt{aguo@mit.edu}
}
\date{}

\begin{document}
\maketitle

\begin{abstract}
We provide a winning strategy for sums of games of \textsc{Mark}-$t$,
an impartial game played on the nonnegative integers where each move
consists of subtraction by an integer between 1 and $t-1$ inclusive,
or division by $t$, rounding down when necessary.  Our algorithm
computes the Sprague-Grundy values for arbitrary $n$ in quadratic
time. This solves a problem posed by Aviezri Fraenkel.
In addition, we characterize the P-positions and N-positions for
the game in mis\`ere play.
\end{abstract}


\section{Introduction}

The impartial game \textsc{Mark}, introduced in~\cite{fraenkel}, is
played on nonnegative integers, where the options of $n$ are $n-1$ and
$\lfloor n/2 \rfloor$. In \emph{normal play}, the first player unable
to move loses. Those integers from which the \textbf{N}ext player to
play has a winning strategy are N-positions, whereas those from which
the \textbf{P}revious player has a winning strategy are
P-positions. As shown in~\cite{fraenkel}, the P-positions and
N-positions for \textsc{Mark} in normal play have an extremely nice
characterization: $n$ is a P-position if and only if its binary
representation has an odd number of trailing 0's.

The \emph{sum} of games is a collection of games such that a player
moves by selecting one of the component games and making a legal move
in it. A player is unable to move when no component game has any move
left. Just knowing the P- and N-positions for the component games is
insufficient for analyzing the positions of the sum. In normal play,
we use the \emph{Sprague-Grundy} function. In Sprague-Grundy theory,
every impartial game in normal play is equivalent to a \textsc{Nim}
heap of some size, called its \emph{Sprague-Grundy value}, or
$g$-value for short~\cite{sprague, grundy}. In particular, a game is a
P-position if and only if its $g$-value is~$0$. The purpose of the
$g$-function is that the $g$-value of a sum of games is equal to the
bitwise XOR or the $g$-values of the component games, which allows us
to compute the outcome of a given sum of games.

The $g$-values of a game can be computed recursively with the
\emph{mex rule}. We define the mex (\textit{m}inimal
\textit{ex}cludant) function as follows: if $S \subseteq \NN =
\{0,1,2,\ldots\}$, then $\mex S = \min (\NN \setminus S)$, i.e. the
least nonnegative integer not in~$S$. We can then compute the
$g$-value of a game as follows. If $u$ is a position of a game with a
set $S_u$ of options, then $g(u) = \mex g(S_u)$. However, computing
$g$-values this way is computationally inefficient for games such as
\textsc{Mark}, since computing $g(n)$ is $O(n)$, which is exponential
in the input length $\log_2 n$.  Fortunately, \cite{fraenkel} gives an
elegant and simple method for computing $g(n)$.  First, note that
$g(n) \in \{0,1,2\}$, since each position has at most $3$ options.
Fraenkel showed that
$$
g(n) =
\begin{cases}
0 & \text{if $n$ has an odd number of trailing 0's in binary} \\
1 & \text{if $n$ has an even number of trailing 0's and an odd number of
  1's in binary}\\
2 & \text{if $n$ has an even number of trailing 0's
  and an even number of 1's in binary}.
\end{cases}
$$
This gives a linear time algorithm for computing $g(n)$, and hence
a linear time algorithm for computing the $g$-value of a sum of games
of \textsc{Mark}.

In~\cite{fraenkel}, \textsc{Mark} was generalized into the family of
games \textsc{Mark}-$t$, parametrized by integer $t \ge 2$. In
\textsc{Mark}-$t$, a player may move from $n$ to any one of
$n-1,n-2,\ldots, n-(t-1), \lfloor n/t \rfloor$. In particular,
\textsc{Mark} is the special case where $t=2$. It has been shown that
subtraction games, both impartial ~\cite{winningways} and
partizan~\cite{partizansub}, in which the amount subtracted is
restricted to constants, are \emph{periodic} in the sense that their
$g$-values are periodic. The importance of periodicity for \emph{octal
  games} is that it implies there is a polynomial-time winning
strategy~\cite{winningways}.  However, for any $t \ge 2$, the
subtraction game \textsc{Mark}-$t$ is
\emph{aperiodic}~\cite{fraenkel}, yet has a polynomial-time algorithm
for determining whether a given position is P or N.  In
Section~\ref{s:normal} of this paper, we complete the picture by
giving a polynomial-time algorithm for computing the Sprague-Grundy
function for \textsc{Mark}-$t$, giving us a polynomial-time winning
strategy for sums of positions of \textsc{Mark}-$t$.

In \emph{mis\`ere play}, the winning condition is reversed, i.e., the
first player unable to move wins rather than loses. The P- and
N-positions of mis\`ere \textsc{Mark}, denoted \textsc{MiMark}, have
been characterized~\cite{fraenkel}. In Section~\ref{s:misere}, we
extend the characterization to general \textsc{Mark}-$t$ in mis\`ere
play, which we denote \textsc{MiMark}-$t$.


\section{Mark-$t$ in normal play}\label{s:normal}

The case $t = 2$ has been dealt with in~\cite{fraenkel}, so we fix $t
\ge 3$.  For notation, let $R(n)$ denote the representation of $n$
written in base~$t$.  We begin by noting that the P-positions of
\textsc{Mark}-$t$ are precisely the \emph{dopey numbers} (numbers with
an odd number of trailing 0's) when written in
base~$t$~\cite{fraenkel}. Building upon this, we have the following
result.

\begin{thm}\label{t:easy}
For $k \in \{0,1,\ldots,t-2\}$, $g(n) = k$ if and only if $R(n)$ has
an odd number of trailing $k$'s.
\end{thm}
\begin{proof}
We prove this by induction on $k$ and $n$. The base case $k = 0$ for
the equivalence is given by the fact that the set of P-positions is
precisely the set of dopey numbers in base~$t$. Now fix $k > 0$ and
suppose the equivalence holds for smaller values of $k$.  We now
induct on $n$. The base case for the reverse implication is given by
$g(k) = k$, since $k$ has options $0,1,\ldots,k-1$ which by induction
have $g$-values $0,1,\ldots,k-1$ respectively, using the mex rule. The
base case for the forward implication is given by $g(0) = 0 \ne k$.

Now suppose $n > k$ and the equivalence holds for smaller values of
$n$. First suppose $R(n)$ has an odd number $r$ of trailing $k$'s.
We have two cases:
\begin{enumerate}[(i)]
\item%
$r > 1$. The options $n-1,n-2,\ldots,n-k$ have $g$-values
  $k-1,k-2,\ldots,0$ respectively, by the inductive hypothesis, and
  for $i \in \{k+1,\ldots,t-1\}$, the option $m = n-i$ has exactly one
  trailing $t+k-i$ preceded by $k-1 \pmod t$, so $R(m-(t-i))$ has
  exactly one trailing $k$, hence $g(m-(t-i)) = k$ and so $g(m) \ne
  k$. Furthermore the option $m = \lfloor n/t \rfloor$ has an even
  number of trailing $k$'s and hence $g(m) \ne n$ by the inductive
  hypothesis.

\item%
$r = 1$. Since $n > k$, the trailing $k$ is preceded by some $j \ne
  k$.  If there are an even number of $j$'s preceding $k$, then the
  argument in case (i) for $\{0,1,\ldots,k-1\} \subset g(S_n)$ still
  goes through. For $i \in \{k+1,\ldots,t-1\}$, $R(n-i)$ ends with
  $(j-1)(t+k-i)$ preceded by an odd number of $j$'s. If $j > k+1$,
  then we can move by subtracting to make the last digit $k$, and if
  $j = k+1$, then we can move by dividing by $t$, making the last
  digit $k$, hence $g(n-i) \ne k$. Finally, $\lfloor n/t \rfloor$ ends
  in zero $k$'s, so its $g$-value is not equal to $k$.

Now suppose there are an odd number of $j$'s preceding $k$. If $j >
k$, then the argument from case (i) shows that $\{0,1,\ldots, k-1\}
\subset g(S_n)$.  If $j < k$, then the only part where this does not
work is when we move to $n-(k-j)$, but then $g(\lfloor n/t \rfloor) =
j$. For $i \in \{k+1,\ldots,t-1\}$, the argument from case (i) goes
through, where $n-i$ ends in $t+k-i$ preceded by $j-1 \ne t+k-i$,
hence $g(n-i) \ne k$. Finally, we already established that $g(\lfloor
n/t \rfloor) = j \ne k$.

\end{enumerate}

Conversely, suppose $R(n)$ has an even number $r$ of trailing
$k$'s. If $r > 0$, then $m = \lfloor n/t \rfloor$ has an odd number of
trailing $k$'s, which by our inductive hypothesis implies $g(m) = k$,
hence $g(n) \ne k$. Therefore, we consider the case $r = 0$.  Write
$R(n) = \ldots d_1 k^j d_2$, where $d_1,d_2 \ne k$ and $d_1$ is
possibly empty.  We have two cases:
\begin{enumerate}[(i)]
\item%
$j = 0$. If $d_2 > k$, then $R(n - (d_2 - k)) = \ldots d_1k$ and by
  our inductive hypothesis $g(n-(d_2-k)) = k$, hence $g(n) \ne k$.  If
  $d_2 < k$, then we have two sub-cases depending on whether $d_1$ and
  $d_2$ are distinct or not. If $d_1 \ne d_2$, then $R(n)$ has exactly
  one trailing $d_2$, and so $g(n) = d_2 \ne k$. If $d_1 = d_2 < k$,
  then $R(n-(t+k-d_2)) = \ldots (d_1-1)k$ which has exactly one
  trailing $k$ and hence $g(n-(t+k-d_2)) = k$, so $g(n) \ne k$.

\item%
$j > 0$. In this case, we have two sub-cases depending on the parity
  of~$j$.  If $j$ is odd, then $R(\lfloor n/t \rfloor)$ has an odd
  number of trailing $k$'s, and so by the inductive hypothesis
  $g(\lfloor n/t \rfloor) = k$, hence $g(n) \ne k$. If $j$ is even,
  then we have two further sub-sub-cases, depending on whether $d_2 <
  k$.  If $d_2 < k$, then by our inductive hypothesis $g(n) = d_2$. If
  $d_2 > k$, then $R(n - (d_2 - k))$ has an odd number $j+1$ of
  trailing $k$'s, and hence $g(n-(d_2-k)) = k$, so $g(n) \ne k$.

\end{enumerate}
\end{proof}

Note that Theorem~\ref{t:easy} does not hold for $k = t-1$. The proof
breaks down, for example, when showing that $R(n)$ has an odd number
of trailing $(t-1)$'s implies $g(n) = t-1$.  Certainly,
$\{0,1,\ldots,t-2\} \subset g(S_n)$ but it is not clear that $t-1
\notin g(S_n)$.

It remains to distinguish between numbers with $g$-values in $\{t-1,
t\}$.  We begin with the following observation.

\begin{lemma}\label{l:delete}
If $R(n) = w k(t-1)^r$ and $R(m) = w k(t-1)$, where $k \ne t-1$, $r >
1$, and $w$ is a (possibly empty) string, then $g(n) = g(m)$ if and
only if $r$ is odd. In other words, deleting extra trailing $(t-1)$'s
beyond the first alternates the $g$-value between $t-1$ and $t$ for
each $(t-1)$ deleted.
\end{lemma}
\begin{proof}
By induction, it suffices to show that if $R(m) = wk(t-1)^{r-1}$, then
$g(n) \ne g(m)$. This is easy since $m$ is an option of $n$ (by
dividing by $t$).  Since both of these have $g$-values in $\{t-1,t\}$,
the $g$-values alternate.
\end{proof}

Lemma~\ref{l:delete} allows us to delete any trailing $(t-1)'s$ beyond
the first when we are trying to distinguish between numbers with
$g$-values in $\{t-1,t\}$, so we only need to worry about the cases
where the number of trailing $(t-1)$'s is $\le 1$.

\begin{thm}\label{t:hard}
There is a quadratic-time algorithm for computing $g(n)$ if $n$ ends
in a single $t-1$ or in a positive even number of $k$'s for some $k
\ne t-1$.
\end{thm}
\begin{proof}
Our algorithm is recursive.
The base cases are given by $g(n) = t-1$ whenever $R(n) = t-1$ or
$R(n) = k(t-1)$ for some $k < t-1$.
Let $\ell$ be the length of $R(n)$.  We have three cases. The first
two cases correspond to $R(n) = wi^rk(t-1)$ for $r > 0$, depending on
if $i \ge k$ or $i < k$. The third case corresponds to $R(n) =
wk^{2j}$ with $j > 0$.

\begin{enumerate}[(i)]
\item%
$i \ge k$: Suppose $i > k$.
Consider the following sequence of moves (positions written in $t$-ary):
$$ 
wi^rk(t-1) \to wi^rk^2 \to wi^r(k-1)(t-1) \to wi^r(k-1)^2 \to
\cdots \to wi^r0^2 \to wi^{r-1}(i-1)(t-1)^2.
$$
By making this sequence of moves, we stay in positions with
$g$-values in $\{t-1,t\}$, so the $g$-values must alternate. A simple
parity check shows that the $g$-values of the initial and final
positions in the sequence match.  By Lemma~\ref{l:delete}, the final
position's $g$-value disagrees with that of
$wi^{r-1}(i-1)(t-1)$. Furthermore, the length of $wi^{r-1}(i-1)(t-1)$
is $\ell - 1$. We can then recursively run on algorithm on
$wi^{r-1}(i-1)(t-1)$, which is back to case (i) with length~$\ell -
1$.

If $i = k$, the above still works if $r$ is even. If $r$ is odd, then
our initial string was $wk^{r+1}(t-1)$, and so its $g$-value disagrees
with that of $wk^{r+1}$, on which we can recursively run our algorithm
in case (iii) with input length $\ell - 1$.

\item%
$i < k$: Consider the following sequence of moves (positions written
  in $t$-ary):
$$
wi^rk(t-1) \to wi^rk^2 \to wi^r(k-1)(t-1) \to wi^r(k-1)^2 \to \cdots \to wi^{r+1}(t-1)
$$
By making this sequence of moves, we stay in positions with
$g$-values in $\{t-1,t\}$, so the $g$-values must alternate. A simple
party check shows that the $g$-values of the initial and final
positions in the sequence match.  Note that we can move to either
$wi^{r+2}$ and $wi^{r+1}$ from the final position. If $r$ is odd, then
$wi^{r+1}$ has an even number of trailing $i$'s, and the $g$-value of
our initial position disagrees with that of $wi^{r+1}$, which we can
find by recursively running on algorithm in case (iii) with input
length $\ell-1$.  If $r$ is even, then we do the same thing except
with $wi^{r+2}$ which is case (iii) with input length $\ell$.

\item%
$R(n) = wk^{2j}$: Note that from this position we can move to
  $wk^{2j-2}(k-1)(t-1)$, whose $g$-value must differ from that of $n$.
  If $k \ne 0$, then this is just case (i) with input length $\ell$
  and we can recurse.  If $k = 0$, then $R(n) = ui0^{2j}$ for some $i
  > 0$. Then we can move to $u(i-1)(t-1)^{2j}$, which switches the
  $g$-value. By Lemma~\ref{l:delete}, deleting until we have
  $u(i-1)(t-1)$ switches the $g$-value back, and we have case (i) or
  (ii) with input length $\ell - (2j-1)$.
\end{enumerate}

For the time complexity, it is straightforward to verify that each
recursive call runs in $O(\ell)$ time regardless of the case, so it
suffices to show that the algorithm terminates after $O(\ell)$
recursive calls. From (i), the input length decreases. From (ii), we
go to (iii). From (iii), we go to (i) or we go to (ii) with decreased
input length. Therefore, the input length is guaranteed to decrease
after every $2$ recursive calls, and so there are at most $2\ell =
O(\ell)$ recursive calls.

\end{proof}

\begin{cor}
There is a quadratic time algorithm for computing $g(n)$ for any $n$
in \textsc{Mark}-$t$.
\end{cor}
\begin{proof}
Use Theorem~\ref{t:easy} if $R(n)$ has an odd number of trailing
$k$'s, otherwise delete the $j$ extra $(t-1)$'s beyond the first and
use Theorem~\ref{t:hard}, flipping the result if $j$ is odd.
\end{proof}


\section{Mis\`ere Mark-$t$}\label{s:misere}

Let $D$ denote the set of dopey binary numbers, numbers whose binary
representations end in an odd number of 0's, and let $V$ denote their
complement, the \emph{vile numbers} (numbers whose binary
representations end in an even number of 0's). If we swap the powers
of $2$ in these sets to construct new sets $D'$ and $V'$, that is,
\begin{eqnarray*}
D' &=& (D \setminus \{2^{2k+1} : k \ge 0 \}) \cup \{2^{2k} : k \ge 0 \} \\
V' &=& (V \setminus \{2^{2k} : k \ge 0\}) \cup \{2^{2k+1} : k \ge 0\},
\end{eqnarray*}
then it is shown in~\cite{fraenkel} that the set of P- and N-positions for
\textsc{MiMark} are precisely $D'$ and $V'$ respectively. In this
section, we generalize this result to \textsc{MiMark}-$t$.

Let $D_t $ denote the set of dopey numbers in base~$t$, and let $V_t$
denote the set of vile numbers in base~$t$. Define
\begin{eqnarray*}
D'_t &=& (D_t \setminus \{t^{2k+1} : k \ge 0 \}) \cup \{t^{2k} : k \ge 0 \} \\
V'_t &=& (V_t \setminus \{t^{2k} : k \ge 0\}) \cup \{t^{2k+1} : k \ge 0\}.
\end{eqnarray*}

\begin{thm}
The P- and N-positions for \textsc{MiMark}-$t$ are precisely $D'_t$
and $V'_t$, respectively.
\end{thm}
\begin{proof}
It suffices to show that: I. A player moving from any position in
$D'_t$ always lands in a position in $V'_t$; II. Given any position in
$V'_t$, there exists a move to a position in $D'_t$.
\begin{enumerate}[I.]
\item%
Let $d \in D'_t$. We have two cases:

\begin{enumerate}[(i)]
\item%
$R(d) = wi0^{2k+1}$, where $w$ is a (possibly empty) $t$-ary string
  and $i > 0$.  All subtracting moves result in the form
  $w(i-1)(t-1)^{2k}j$ for some $j = 1,2,\ldots,t-1$, which lies in
  $V'_t$. The division move results in $wi0^{2k}$, which also lies in
  $V'_t$.

\item%
$R(d) = 10^{2k}$. The base case $k = 0$ is true since $1$ can only
  move to $0$, which is an N-position, so assume $k > 0$. Then any
  subtraction move results in the form $0(t-1)^{2k-1}j$ for some $j =
  1,2,\ldots,t-1$, which lies in $V'_t$. The division move results in
  $10^{2k-1}$, which also lies in $V'_t$.

\end{enumerate}

\item%
Let $v \in V'_t$. We again have two cases:

\begin{enumerate}[(i)]
\item%
$R(v) = wi0^{2k}$, where $w$ is a (possibly empty) $t$-ary string and
  $i > 0$.  If $k > 0$, we can divide by $t$ to move to $wi0^{2k-1}$,
  which lies in $D'_t$, so suppose $k = 0$, i.e. $R(v) = wi$. If $w$
  does not end with 0, then we can subtract by $i$ to move to $w0$
  which lies in $D'_t$, so suppose $w = u0^r$ and hence $R(v) =
  u0^ri$. If $r$ is even, we can subtract by $i$ to move to
  $u0^{r+1}$, which lies in $D'_t$. If $r$ is odd, we can divide by
  $t$ to move to $u0^r$, which lies in $D'_t$.

\item%
$R(v) = 10^{2k+1}$. Dividing by $t$ moves us to $10^{2k}$, which lies
  in $D'_t$.

\end{enumerate}

\end{enumerate}
\end{proof}



\end{document}